\documentclass{amsart}
\usepackage{graphicx}
\usepackage{amsmath}
\usepackage{amssymb}
\usepackage{amsthm}
\usepackage{esint}
\usepackage{color}

\let\a\alpha
\let\b\beta
\let\d\delta

\let\e\epsilon

\let\g\gamma

\let\k\kappa

\let\m\mu
\let\n\nu
\let\o\omega
\let\r\rho

\newtheorem{thm}{Theorem}[section]
\newtheorem{problem}{Problem}[section]
\newtheorem{lemma}{Lemma}[section]
\newtheorem{coro}{Corollary}[section]

\newtheorem{rem}{Remark}[section]
\newenvironment{ack}[1][Acknowledgement]{\noindent\textbf{#1.} }{\ }
\newenvironment{mt}[1][Main Theorem]{\noindent\textbf{#1.} }{\ }
\numberwithin{equation}{section}
\newcommand{\R}{{\mathbb R}}

\numberwithin{equation}{section}

\begin{document}

\title{Area-preserving mean curvature flow of rotationally symmetric hypersurfaces with free boundaries}


\author{Kunbo Wang}
\address{School of Mathematical Sciences, Zhejiang University, Hangzhou 310027, China.}
\email{21235005@zju.edu.cn}


\date{December 16, 2017}


\begin{abstract}
In this paper, we consider the area-preserving mean curvature flow
with free Neumann boundaries. We show that for a rotationally
symmetric $n$-dimensional hypersurface in $\R^{n+1}$ between two
parallel hyperplanes will converge to a cylinder with the same area
under this flow. We use the geometric properties and the maximal
principle to obtain gradient and curvature estimates, leading to
long-time existence of the flow and convergence to a constant mean
curvature surface.
\end{abstract}

\maketitle

\section{Introduction and the main results}
A Hypersurface $M_t$ in Eucilidean space is said to be evolving by
mean curvature flow if each point $X(\cdot)$ of the surface moves,
in time and space, in the direction of its unit normal with speed
equal to the mean curvature $H$ at that point. That is
\begin {equation} \label{1.1}
 \frac{\partial X}{\partial t}=-H\nu(x,t),
\end {equation}
where $\nu(x,t)$ is the outer unit normal. It was first studied by
Brakle in [6] from the viewpoint of geometric measure theory. In
[12], G. Huisken showed that any compact, convex hypersurface
without boundary converges asymptotically to a round sphere in a
finite time interval. Mean curvature flow is also the steepest
descent flow for the area functional, evolving to minimal surfaces.
In [13], Huisken initiated the idea of considering the following
volume-preserving mean curvature flow,
\begin {equation} \label{1.2}
\frac {\partial}{\partial t}X(x,t)=(h(t)-H(x,t))\nu(x,t),
\end {equation}
 where $h(t)=\frac{\int
_{M_t} H d\mu_t}{\int _{M_t}d\mu_t}$ is the average of the mean
curvature on $M_t$. Huisken proved if the initial hypersurface
$M_0^n$ $(n\geq 2)$ is uniformly convex, then the evolution equation
\eqref{1.2} has a smoothly solution $M_t$ for all times $0\leq t<
\infty$ and $M_t$ converges to a round sphere enclose the same
volume as $M_0$ in the $C^{\infty}$-topology as $t\rightarrow
\infty$. In [20], Pihan studied the following area-preserving mean
curvature flow.
\begin {equation} \label{1.3}
\frac {\partial}{\partial t}X(x,t)=(1-h(t)H(x,t))\nu (x,t).
\end{equation}
Here $h(t)=\frac{\int _{M_t} H d \mu _t}{\int_{M_t} H^2 d \mu _t}$.
Pihan showed that if the initial hypersurface is compact without
boundary, \eqref{1.3} has a unique solution for a short time under
the assumption $h(0)>0$. For $n=1$, Pihan also showed  that an
initially closed, convex curve in the plane converges exponentially
to a circle with the same length as the initial curve. For $n\geq
2$, Mccoy in [16] showed that if the initial $n$-dimensional
hypersurface $M_0$ is strictly convex£¬ then the evolution equation
\eqref{1.3} has a smooth solution $M_t$ for all time $0\leq t <
\infty$, and $M_t$ converge, as $t\rightarrow \infty$, in the
$C^{\infty}$-topology, to a sphere with the same surface area as
$M_0$. In [11], Huang and Lin use the idea of iteration of Li in
[14] and Ye in [23], in cases of volume preserving mean curvature
flow and Ricci flow, respectively. And obtained the same result, by
assuming that the initial hypersurface $M_0$ satisfies $h(0)>0$ and
$\int_{M_0}|A|^2-\frac{1}{n}H^2d\m\leq \e.$ \\

In this paper, we study the area preserving mean curvature flow with
free boundaries, where a restriction on the angles of boundaries
with fixed hypersurfaces in Euclidean space are imposed. In this
setting, there are some interesting works (cf. [7], [19] and [22]).
In these papers, the authors study the mean curvature flow with
Neumann and Dirichlet free boundaries. Let $\Sigma$ be a fixed
hypersurface smoothly embedded  in $\mathbb{R}^{n+1}$. We say
$X(x,t)$ is evolved by mean curvature flow with free Neumann bounary
condition on $\Sigma$, if
\begin{eqnarray*}
\frac{\partial X}{\partial t}&=&-H\nu,\quad \forall(x,t)\in M^n\times [0,T),\\
\langle\n_{M_t},\n_{\Sigma}\circ X\rangle(x,t)&=&0, \quad
\forall(x,t)\in
\partial M^n\times [0,T),\\
 X(\cdot,0)&=&M_0,\\
X(x,t)&\subset& \Sigma,\quad \forall(x,t)\in \partial M^n\times
[0,T).
\end{eqnarray*}
The volume-preserving mean curvature flow with free Neumann
boundaries was first studied by Athanassenas in [2].

Let $M_0$ be a complete $n$-dimensional hypersurface with boundary
$\partial M_0 \neq \emptyset. $ Assume $M_0$ is smoothly embedded in
the domain
$$ G=\{x \in \R^{n+1}: 0<x_{n+1}<d, d>0 \},$$
We denote by $\Sigma_i(i=1,2),$ the two parallel hyperplanes
bounding the domain $G$, and assume $\partial M_0 \subset
\Sigma_i(i=1,2)$. Then Athanassenas proved the following theorem in
[2].
\begin {thm}
Let $V,d \in \R$ be given two positive constants. $M_0 \subset G$ is
a smooth, rotationally symmetric, initial hypersurface which
intersects $\Sigma_i(i=1,2)$ orthogonally at the boundaries which
encloses the volume $V$. Then the free Neumann boundaries problem
for equation \eqref{1.2} has a unique solution on $[0, +\infty)$,
which converges to a cylinder $C \subset G$ of volume $V$ under
assumption $|M_0| \leq \frac {V}{d}$ as $t\rightarrow \infty$.
\end {thm}

Other works on this problem were investigated in [3], [4], [17] and
[18]. In this paper, we consider the following problem for
area-preserving mean curvature flow with free Neumann boundaries.
\begin{problem}
\begin{eqnarray*}
\frac {\partial}{\partial
t}X(x,t)&=&(1-h(t)H(x,t))\nu (x,t),\quad \forall(x,t)\in M^n\times [0,T),\\
\langle\n_{M_t},\n_{\Sigma_i}\circ X\rangle(x,t)&=&0, \quad
\forall(x,t)\in
\partial M^n\times [0,T),i=1,2,\\
 X(\cdot,0)&=&M_0,\\
X(x,t)&\subset& \Sigma_{i},\quad \forall(x,t)\in \partial M^n\times
[0,T), i=(1,2).
\end{eqnarray*}
\end{problem}
We prove the following main theorem for Problem 1.1.\\

\begin{mt} Let $V,d \in \R^{+}$  be given two constants and $M_0 \subset
G$ to be a smooth, rotationally symmetric, mean convex initial
hypersurface which intersects $\Sigma_i(i=1,2)$ orthogonally at the
boundaries. Then the solution to Problem 1.1 exists for all times
$t>0$ and converges to a cylinder of the same area with $M_0$ under
the assumption $|M_0|\leq \frac{V}{d}$.
\end{mt}
\begin {rem}
We say $M_0$ is mean convex if the mean curvature is positive
everywhere. The condition of mean convex will be used to prove the
equation \eqref{1.3} is strictly parabolic. In [20], Pihan shows
that the equation \eqref{1.3} is strictly parabolic for a short time
if $h(0)>0$. And as the case of volume-preserving mean curvature
flow in [3], Problem 1.1 is  a Neumann boundary problem for strictly
parabolic equation, from which we obtain the short time existence.
And also see [19] and [22] for details of general cases.
\end {rem}
This paper is organized as follows. In Section 2, we give some
definitions and preliminaries. In Section 3, we show some basic
properties of this flow. We prove that the property of mean
convexity can be preserved under equation \eqref{1.3} and the
surfaces do not pinch off under the condition of our main theorem.
In Section 4, we use the property of mean convexity and maximal
principle to show the curvature estimates. Gradient and curvature
estimates lead to long time existence to a constant curvature
surface. And we prove our main theorem in Section 5.\\
The methods we use here are those introduced by Athanassenas in [2],
Ecker and Huisken in [9]. We use the free Neumann boundary condition
to convert the boundary estimates to interior estimates (see Lemma
3.4, Lemma 4.1 and Theorem 4.1). We put the condition of mean
convexity here is to give an upper and lower bounds for $h(t)$ and
$v(x,t)$, which is crucial for our curvature estimates.

\section{Preliminaries }

We adopt the similar notations of Huisken in [12].  Let $M$ be an
$n$-dimensional Riemannian manifold. Vectors on $M$ are denoted by
$X=\{ X^i\}$, covectors by $Y=\{Y_i\}$ and mixed tensors by $T=\{
T_j^{ik}\}$. The induced metric and the second fundamental form on
$M$ are denoted by
 $g=g_{ij}$ and $A=\{ h_{ij}\}$ respectively. The surface area
 element of $M$ is given by
 $$\mu =\sqrt{det(g_{ij})}.$$
 For tensors $T_{ijkl}$ and $U_{ijkl}$ on M, we have the inner product
$$(T_{ijkl},U_{ijkl})=g^{i\a}g^{j\b}g^{k\g}g^{l\d}T_{ijkl}U_{\a\b\g\d}$$
and the norm
$$|T_{ijkl}|^2=(T_{ijkl},T_{ijkl}).$$
The trace of the second fundamental form,  $H=g^{ij} h_{ij},$ is the
mean curvature of $M$, and $ |A|^2=g^{ik}g^{jl}h_{ij}h_{kl}$ is the
square of the norm of the second fundamental form on $M$. We also
denote
$$\tilde{C}=tr(A^3)=g^{ij}g^{kl}g^{mn}h_{ik}h_{jm}h_{ln}.$$
If $X: M^n\hookrightarrow \R^{ n+1}$ smoothly embeds $M^n$ into
$\R^{n+1}$, then the induced metric $g$ is given by
$g_{ij}=\langle\frac{\partial X}{\partial x_i}(x),\frac{\partial
X}{\partial x_j}(x)\rangle$ and the second fundamental forms
$h_{ij}=\langle\frac{\partial X}{\partial x_i}(x),\frac{\partial
\nu}{\partial x_i}(x)\rangle$, where $\langle\cdot , \cdot\rangle$
is the ordinary scalar product of vectors in $\R^{n+1}.$ The matrix
of the Weingarten map of $M$ is  $h_j^i(x)=g^{ik}(x)h_{kj}(x).$ The
eigenvalues of this matrix are the principal curvatures of $M$. The
induced connection on $M$ is given via the Christoffel Symbols.
$$\Gamma_{ij}^{k}=\frac{1}{2}g^{kl}(\frac{\partial g_{jl}}{\partial x^i}+\frac{\partial g_{il}}{\partial x^j}
-\frac{\partial g_{ij}}{\partial x^l}).$$ The covariant derivative,
for a vector field $v=v^j \frac{\partial}{\partial x^j}$ is given by
$$\nabla _i v^j=\frac{\partial v^j}{\partial x^i }+\Gamma _{ik}
^j v^k,$$
$$\nabla _i v_j=\frac{\partial v_j}{\partial x^i }-\Gamma _{ik}
^j v_k.$$ The Laplacian of T is
$$\triangle T=g^{ij}\nabla _i \nabla_j T.$$
The Riemannian curvature tensor on $M$ can be given through the
Gauss Equations
$$R_{ijkl}=h_{ik}h_{jl}-h_{il}h_{jk}.$$
We denote $|M_t|$ to be the surface area of $M_t$.  We assume $M_0$
is mean convex and rotationally symmetric about an axis which
intersects $\Sigma _i$ orthogonally. We denote this axis by
$x_{n+1}$ and use the parametrization
$$\rho(x_{n+1}):[0,d]\mapsto \R$$
for the generating curve of a surface of revolution, which is a
radius function.

\section{ Basic properties}
From now on,  we write $[0,T)$ to indicate the maximal time interval
for which the flow exists. First we verify that the surface area
does indeed remain fixed under the area-preserving flow \eqref{1.3},
while the enclose volume does not decrease. The rotationally
symmetric property is preserved under the equation \eqref{1.3}. This
is clear from the evolution equation, since the mean curvature and
the normal are  symmetric.
\begin{lemma}
The surface area of $M_t$ remains constant throughout the  flow,
that is
$$\frac{d}{dt}\int _{M_t}d\mu _t\equiv 0.$$
\end {lemma}
\begin{proof} We apply the first variation of area formula to the vector
field ${\frac{\partial X}{\partial t}}$, extended appropriately, and
the divergence theorem,
\begin {align*}
\frac{d}{dt}\int _{M_t}d\mu_t=\int_{M_t} div_{M_t}(\frac {\partial
X}{\partial t})d \mu_t   = - \int_{M_t} (1-hH)H d\mu_t\equiv 0.
\end {align*}
\end {proof}
\begin{lemma}
The volume enclosed by $M_t$ does not decrease throughout the flow.
That is, if $E_t\subset \R^{n+1}$ is the $(n+1)$-dimensional set
enclosed by $M_t$ and the two parallel planes  $\Sigma _i$, then
$$\frac{d}{dt}Vol(E_t)\geq 0$$
\end {lemma}
\begin{proof} We first extend $\frac{\partial
X}{\partial t}$ to a vector field on the whole of $E_t$, then apply
the first variation of area formula and divergence theorem,
\begin{eqnarray*}
\frac{d}{dt}Vol(E_t)&=&\int_{E_t} div_{\R^{n+1}}(\frac{\partial
X}{\partial t})dV=\int_{\partial E_t}<\frac{\partial X}{\partial
t}\nu>d\mu_t\\& =&\int_{M_t}<\frac{\partial X}{\partial
t}\nu>d\mu_t+\int_{\Sigma_i}<\frac{\partial X}{\partial
t}\nu>d\mu_t\\&=& \int_{M_t}d\mu_t-\frac{(\int_{M_t} H
d\mu_t)^2}{\int_{M_t} H^2 d\mu_t}\geq 0.
\end{eqnarray*}
Here we have use the free Neumann boundary condition to obtain the
integral on $\Sigma _i$ is $0$.
\end{proof}
As in [16] (Section 4), we have the following evolution equations.
\begin {lemma}
We have
\\(1) $\frac{\partial}{\partial t}g_{ij}=2(1-hH)h_{ij};$
\\(2) $\frac{\partial}{\partial t}g^{ij}=-2(1-hH)h^{ij}$;
\\(3) $\frac{\partial}{\partial t}\n=h\nabla H$;
\\(4) $\frac{\partial}{\partial t} h_{ij}=h\triangle h_{ij}+(1-2hH)h^m_i h_{mj}+h|A|^2h_{ij};$
\\(5) $\frac{\partial}{\partial t}h_j^i=h(\triangle h^i_j+|A|^2h^i_j)-h^i_m h^m_j;$
\\(6) $\frac{\partial}{\partial t} H=h\triangle H -(1-hH)|A|^2;$
\\(7) $\frac{\partial}{\partial t}|A|^2=h(\triangle |A|^2-2|\nabla A|^2+2|A|^4)-2\tilde{C};$
\\(8) $(\frac{\partial}{\partial t}-h\triangle)H^2=-2h|\nabla H|^2-2(1-hH)H|A|^2.$
\end {lemma}
\begin {lemma}
The mean curvature is positive on $M_t$, $t\in [0,T)$. Furthermore
on the boundaries $\partial M_t$, we have $\lim
\limits_{x_{n+1}\rightarrow 0} H(x,t)=a(t)>0$, $\lim
\limits_{x_{n+1}\rightarrow d}H(x,t)=b(t)>0.$
\end {lemma}
\begin{proof} Since we consider the hypersurface has free boundaries, we can
not directly use the maximal principle. Suppose the first time
$H(x,t)=0$ is attained at an interior point of $M_t$, then from
Lemma 3.3, we have
$$(\frac{\partial}{\partial t}-h\triangle)H^2=-2h|\nabla H|^2-2(1-hH)H|A|^2.$$

From the maximal principle, we know $H(x,t)=0$ can not be attained
at an interior point of $M_t$. If $t_0$ is the first time such that
$\lim \limits_{x_{n+1}\rightarrow 0}H(x,t_0)=0$ or $\lim
\limits_{x_{n+1}\rightarrow d}H(x,t_0)=0$. By a reflection of
$\Sigma_1$ and $\Sigma_2$, we can define two pieces of new
hypersurfaces outside the boundary, which satisfies the free Neumann
boundaries conditions. Denote $\tilde{M_{t_0}}$ to be the new
hypersurface and $\tilde{\r}(x_{n+1})$ its radius function .
Precisely,
\[
\tilde{\r}(x_{n+1}) = \left\{
\begin{array}
{c @{\quad , \quad} c}
\r(2d-x_{n+1}) & d\leq x_{n+1}\leq 2d\\
\r(x_{n+1}) & 0\leq x_{n+1}< d\\
\r(-x_{n+1}) & -d\leq x_{n+1}<0
\end{array}
\right.
\]
i.e. $\tilde{H}(x_1,x_2,\cdots,x_{n},0,t_0)=\lim
\limits_{x_{n+1}\rightarrow 0}H(x,t_0)$,
$\tilde{H}(x_1,x_2,\cdots,x_{n},d,t_0)=\lim
\limits_{x_{n+1}\rightarrow d}H(x,t_0)$. Then at the boundary
points, $\frac{\partial}{\partial t}\tilde{H}(x,t_0)\leq0$,
$\triangle \tilde{H}(x,t_0)> 0$. So the maximal principle can still
be applied, which proves the lemma.
\end{proof}

Now we will show that the radius of the hypersurface $M_t$ has
uniform lower and upper bounds for any time $t\in [0,T)$. The method
follows from [2](Lemma 1).
\begin {lemma}
Under the conditions of the Main Theorem, there exist constant $r$
and $R$ only depending on $n$,$d$,$V$ and $|M_0|$£¬such that $r\leq
\r(x_{n+1},t)\leq R$ for any $t\in [0,T)$.
\end {lemma}
\begin{proof} Given an initial surface $M_0$, we denote by $C$ the cylinder
with the same enclosed volume $V$  as $M_0$ in $G$. Assume that
there is some $t_0>0$ such that $M_{t_0}$ pinches off. We project
$M_{t_0}$ onto the plane $\Sigma_1$, using the natural projection
$\pi : \R^{n+1}\rightarrow \R^{n}$. Then
$$|M_{t_0}|\geq |\pi(M_{t_0})|.$$
Any $M_t$ has to intersect the cylinder $C$ at least once by the
volume constrain, that the volume of $M_t$ is not decreasing.
Therefore
$$|M_0|=|M_{t_0}|>|\pi (M_{t_0})|>|\pi (C)|=\omega _n \r_{C}^n=\frac{V}{d}.$$
Here $\r_{C}$ is the radius of $C$, and $\omega_n$ is the volume of
unit ball in $\mathbb{R}^n$, then we obtain a contradiction. For the
upper bound, we assume that $\r(x_{n+1},t)_{max}=R(t)$, then we have
$$|M_0|=|M_t|>\omega_n\cdot (R(t)-\r_C)^n,$$ which implies
$$R(t)<\r_C+(\frac{|M_0|}{\omega_n})^{\frac{1}{n}}.$$
\end{proof}

\section {Curvature estimates}

Let $\hat{x}=(x_1,\cdots,x_n,0)$, and
$\omega=\frac{\hat{x}}{|\hat{x}|}$ to denote the unit outer normal
to the cylinder intersecting $M_t$ at the point $X(P,t)$. As in [2]
and [21], we define the height function of $M_t$ to be $u=\langle
X,\omega\rangle$. And define
$v(x,t)=\langle\omega,\nu\rangle^{-1}=\sqrt{(\dot{\r})^2+1}$, where
$\dot{\r}$ is the derivative of $\r$ about $x_{n+1}$. From Lemma
3.5, we can obtain the height estimate $r\leq u(x,t)\leq R.$ Now we
show that $v$ has an upper bound under the assumption $T<\infty$.
The kernel ideal is that according to our parametrization, the
points where  $v$ tends to infinity are not the singular points of
the evolving surface, and the $|A|$ tends to zero at these points.

\begin {lemma} If $T<\infty$, then $v(x,t)\leq M_T<+\infty$£¬ for
any $t\in [0,T]$, in the limitation sense when $t=T$, i.e.
$v(x,T)=\lim \limits_{t\rightarrow T}v(x,t)$. Here $M_T$ is a
constant depending on $n,d,T,r,R,V$ and $|M_0|$.
\end {lemma}
\begin{proof} Since $M_t$ is rotationally symmetric we have
$H=\kappa_1+(n-1)\kappa_2,$ where $\kappa_1$ and $\k_2$ denote the
principle curvatures. If we parameterize $M$  by its radius function
$\r \in C^{\infty}([0,d])$, then clearly
$$H=\frac{-\ddot{\r}}{(1+\dot{\r}^2)^{\frac{3}{2}}}+\frac{n-1}{\r(1+\dot{\r}^2)^{\frac{1}{2}}}.$$
Suppose $t_0$ is the first time such that $\lim
\limits_{x_{n+1}\rightarrow s}v(x,t_0)=\lim
\limits_{x_{n+1}\rightarrow
s}\sqrt{(\dot{\r})^2(x_{n+1},t_0)+1}=+\infty$ for some $s\in (0,d)$.
Since $\lim \limits_{x_{n+1}\rightarrow
s}\dot{v}=\frac{1}{2}\frac{2\dot{\r}\ddot{\r}}{\sqrt{(\dot{\r})^2+1}}=0$,
we have $\lim \limits_{x_{n+1}\rightarrow
s}\ddot{\r}(x_{n+1},t_0)=0$, then we have $H=0$ at this point, which
is a contradiction with $H>0$ everywhere. If $\lim
\limits_{t\rightarrow T,x_{n+1}\rightarrow s}v(x,t)=+\infty$, then
$\lim \limits_{t\rightarrow T,x_{n+1}\rightarrow
s}|A|^2(x,t)=\frac{(\ddot{\r})^2(x_{n+1},t)}{[(\dot{\r})^2(x_{n+1},t)+1]^3}+\frac{n-1}{\r^2[(\dot{\r})^2(x_{n+1},t)+1]}=0$
which implies $X(x_1,\cdots,x_n,s,T)$ is not a singular point. So
the maximal principal method in Lemma 3.4 can still be applied and
we must have $H>0$ at this point, which is still a contradiction.
The boundary points case can be proved in the same way as in Lemma
3.4. Now we show that $v(x,t)_{max}<+\infty$ for all $t\in [0,T]$,
by the continuity of $v$, we have $v\leq M_T$ for some constant
depending on $T$.
\end{proof}
Next, we will show an estimate of $h(t)$. First, we prove the
following lemma.
\begin {lemma}
Under the assumption of the main theorem, we have $C_1\leq
\int_{M_t}Hd\m_t \leq C_2$, for all time $t\in [0,T)$. Here $C_1$
and $C_2$ are positive constants only depending on $n,d,r,R,V$ and
$|M_0|$.
\end {lemma}
\begin{proof} First we show that $\int_{M_t}Hd\m_t\geq C_1$ for some constant
$C_1$. This is a direct consequence of the first variation formula
and mean curvature is positive. Since
$$n|M_0|=n|M_t|\leq \int _{M_t}H<X,\n>d\m_t\leq\int _{M_t}H|X|d\m_t\leq\sqrt{d^2+R^2}\int
_{M_t}Hd\m_t,$$  the lower bound for  $\int_{M_t}Hd\m_t$ is
obtained. Next we show there is an upper bound for
$\int_{M_t}Hd\m_t$, that there is a constant $C_2$ such that
$\int_{M_t}Hd\m_t\leq C_2$. We still parameterize $M_t$  by its
radius function $\r(x_{n+1},t) $. We denote $\omega_n$ to be the
volume of unit ball in $\mathbb{R}^n$, and its surface area is
$n\omega_n$, then we have
\begin{eqnarray*}
H&=&\frac{-\ddot{\r}}{(1+\dot{\r}^2)^{\frac{3}{2}}}+\frac{n-1}{\r(1+\dot{\r}^2)^{\frac{1}{2}}}.\\
|M_t|&=& n\omega_n\int ^d _0 \r^{n-1} \sqrt{1+\dot{\r}^2}dx_{n+1}.\\
\int_{M_t} H d \mu_t&=&n\omega_n \int ^d _0
(-\frac{\ddot{\r}}{(1+\dot{\r}^2)}\r^{n-1}+(n-1)\r^{n-2})dx_{n+1}.
\end{eqnarray*}
On one hand, we have $\int ^d _0 (n-1)\r^{n-2}d\m_t\leq
(n-1)dR^{n-2}.$\\
On the other hand, by our boundary conditons,\\
\begin{eqnarray*}
\int ^d _0-\frac{\ddot{\r}}{(1+\dot{\r}^2)}\r^{n-1} dx_{n+1}&=&
\int^d_0 -\r^{n-1}d(\arctan \dot{\r}) dx_{n+1} \\
&=&-\r^{n-1}\cdot(\arctan\dot{\r}|^d_0)+(n-1)\int^d_0 \r^{n-2}\cdot \dot{\r}\cdot \arctan \dot{\r} dx_{n+1}\\
&\leq&(n-1)\frac{\pi}{2}\cdot\int ^d
_0\r^{n-2}\sqrt{1+\dot{\r}^2}dx_{n+1}\\
&\leq& \frac{(n-1)\pi}{2r}\int ^d
_0\r^{n-1}\sqrt{1+\dot{\r}^2}dx_{n+1}\\
&=& \frac{(n-1)\pi}{2n \omega_n} |M_t|
 \frac{1}{r}
\\&=& \frac{(n-1)\pi}{2n \omega_n} |M_0| \frac{1}{r}.
\end{eqnarray*}
Thus the upper bound is obtained.
\end{proof}
\begin{coro}
Under the assumption of the main theorem, we have $0<h(t)\leq C_3$
for all time $t\in [0,T)$. Here $C_3$ is a constant only depending
on $n,d,r,R,V$ and $|M_0|$.
\end{coro}
\begin{proof} Now we use the Cauchy-Schwarz inequality
$$\frac{(\int _{M_t} H d\mu)^2}{(\int _{M_t} H^2 d\mu)(\int_{M_t} d\mu)}\leq 1.$$
From which we have $0<\frac{\int_{M_t} H d\m_t}{\int_{M_t} H^2
d\m_t}\leq \frac{\int_{M_t} d\m_t}{\int_{M_t} H
d\m_t}=\frac{|M_0|}{\int_{M_t} H d\m_t}\leq C_3$.
\end{proof}
Now we show that $|A|^2$ is bounded for any finite time interval.
\begin {thm}
If the maximal time interval $[0,T)$ is finite, i.e. $T<+\infty$,
then we have $|A|^2(x,t)\leq C_T$, where $C_T$ is a constant
depending only on $T$,$n,d,r,R,V$ and $|M_0|$.
\end {thm}

\begin{proof} First we compute the evolution equation of
$v(x,t)=<\o,\n>^{-1}$. Clearly we have
$$\frac{\partial }{\partial t}v=-v^2\langle w,\frac{\partial }{\partial t} \n\rangle=-v^2\cdot\langle h\nabla H,\o\rangle.$$
From [2] , we have $$\triangle v=|A|^2v-v^2<\o,\nabla
H>+\frac{2|\nabla v|^2}{v}-\frac{n-1}{u^2}\cdot v .$$Then we obtain
$$(\frac{\partial }{\partial t}-h\triangle)v=-h|A|^2v-\frac{2h|\nabla
v|^2}{v}+\frac{(n-1)hv}{u^2},$$ and
\begin{eqnarray*}
(\frac{\partial }{\partial
t}-h\triangle)v^2&=&2v\cdot(-h|A|^2v-\frac{2h|\nabla
v|^2}{v}+\frac{(n-1)hv}{u^2})-2h|\nabla v|^2\\&=&-6h|\nabla
v|^2-2hv^2|A|^2+\frac{2(n-1)hv^2}{u^2}.
\end{eqnarray*}
We considering $|A|^2v^2$ as in [12] and divide the points in $M_t$
into three sets.
\begin{eqnarray*}
S_t&=&\{P\in M_t|\ddot{\r}\geq 0\}.\\
I_t&=&\{P\in M_t|\ddot{\r}<0,\frac{\kappa_1}{\k_2} < \alpha \}.\\
J_t&=&\{P\in M_t|\ddot{\r}<0,\frac{\kappa_1}{\k_2} \geq \alpha\}.
\end{eqnarray*}
 Here $\alpha$ is a positive constant large enough. We will show that for
all points in $S_t$ and $I_t$, $|A|^2v^2$ has uniform upper bounds
for any $t\in [0,T)$. We split our proof into three cases.\\

$Case(1).$ If $P\in S_t $, from
$H=\frac{-\ddot{\r}}{(1+\dot{\r}^2)^{\frac{3}{2}}}+\frac{n-1}{\r(1+\dot{\r}^2)^{\frac{1}{2}}}>0$,
we have

$$\ddot{\r}<\frac{n-1}{\r}[1+(\dot{\r})^2].$$
Then we have
\begin{eqnarray*}
|A|^2v^2&=&\frac{(\ddot{\r})^2}{[1+(\dot{\r})^2]^2}+\frac{n-1}{\r^2}\\
&\leq&\frac{(n-1)^2}{\r^2}+\frac{n-1}{\r^2}\leq C.
\end{eqnarray*}
From now on, we denote by $C$ to any constant depending on
$n,V,d,r,R$ and
$M_0$.\\
$Case(2).$ If $P\in I_t$, then $\frac{\kappa_1}{\k_2} < \alpha$. We
have $\frac{-\r\cdot \ddot{\r}}{(\dot{\r})^2+1}<\alpha$, and
$\frac{- \ddot{\r}}{(\dot{\r})^2+1}<\frac{\alpha}{r}$. Thus
$$|A|^2v^2=\frac{(\ddot{\r})^2}{[1+(\dot{\r})^2]^2}+\frac{n-1}{\r^2}\leq C.$$
$Case(3).$ For points in $J_t$, we use the technique of maximal
principle. First we have
\begin{eqnarray*}
(\frac{\partial }{\partial t}-h\triangle)|A|^2v^2&=&|A|^2\cdot
(-6h|\nabla
v|^2-2hv^2|A|^2+\frac{2(n-1)hv^2}{u^2})\\&+&v^2\cdot(-2h|\nabla
A|^2+2h|A|^4-2\tilde {C})-2h\nabla|A|^2\cdot\nabla v^2\\
&=&|A|^2\cdot (-6h|\nabla
v|^2-2hv^2|A|^2+\frac{2(n-1)hv^2}{u^2})\\&+&v^2\cdot(-2h|\nabla
A|^2+2h|A|^4-2\tilde {C})+h(-\nabla |A|^2\nabla
v^2-4v|A|\nabla|A|\nabla v)\\
&=&|A|^2\cdot (-6h|\nabla
v|^2-2hv^2|A|^2+\frac{2(n-1)hv^2}{u^2})\\&+&v^2\cdot(-2h|\nabla
A|^2+2h|A|^4-2\tilde {C})+h[-v^{-2}\nabla v^2 \nabla
(|A|^2v^2)+v^{-2}|\nabla v^2|^2|A|^2\\&-&4v|A|\nabla |A| \nabla v]+\frac{2(n-1)h|A|^2v^2}{u^2}\\
&\leq&-6h|A|^2|\nabla
v|^2-2hv^2|\nabla|A||^2-2\tilde{C}v^2-hv^{-2}\nabla
v^2\nabla(|A|^2v^2)\\&+&4h|\nabla v|^2|A|^2-4hv|A|\nabla|A|\nabla v+\frac{2(n-1)h|A|^2v^2}{u^2}\\
&\leq&-hv^2\nabla
v^2\nabla(|A|^2v^2)+\frac{2(n-1)h|A|^2v^2}{u^2}-2\tilde{C}v^2.
\end{eqnarray*}
We have used $|\nabla |A||\leq |\nabla A|$ and Cauchy-Schwarz
inequality. Since
\begin{eqnarray*}
\frac{2(n-1)h|A|^2}{u^2\tilde{C}}&\leq& C_4\cdot
\frac{\k_1^2+(n-1)\k_2^2}{\k_1^3+(n-1)\k_2^3}\\
&=&C_4\cdot
\frac{\frac{1}{\k_1}+(n-1)\frac{1}{\k_1}(\frac{\k_2}{\k_1})^2}{1+(n-1)(\frac{\k_2}{\k_1})^3}
\\&\leq&C_5\cdot \frac{1}{\k_1}.
\end{eqnarray*}
Then, if $\k_1>C_5$, we have
$\frac{2(n-1)h|A|^2v^2}{u^2}-2\tilde{C}v^2<0$. Thus $|A|^2v^2$ can
not attain a maximal value by the maximal principle. And if
$\k_1\leq C_5$, we have
$$|A|^2=k_1^2+\frac{n-1}{\r^2[1+(\dot{\r})^2]}\leq C,$$
and $|A|^2v^2\leq C M_T^2$. Therefore, $|A|^2\leq \frac{C
M_T^2}{v^2_{min}}=C_T$.
\end{proof}
\begin {coro}
Under the assumption of the above theorem, we have a lower bound for
$h(t)$, namely, $h(t)\geq m_T$. Here, $m_T$ is a constant depending
on $T$,$n$,$V$,$d,r,\alpha,R$ and $|M_0|$.
\end {coro}
\begin{proof} It is a direct consequence of $H^2\leq n |A|^2$, Lemma 4.2 and
Theorem 4.1.
\end{proof}
Next we give the higher derivative estimates as Hamilton in [10].
\begin {coro}
Under the assumption of Theorem 4.1, we have the following higher
derivative estimates
$$|\nabla^m A|^2\leq C_m(T)$$
\end {coro}
\begin{proof} First, we have
\begin{eqnarray*}
 \frac{\partial}{\partial
t}|\nabla^mA|^2&=&h\triangle|\nabla^mA|^2-2h|\nabla^{m+1}A|^2+\sum_{i+j+k=m}\nabla^iA\ast\nabla^jA\ast\nabla^kA\ast\nabla^mA
\\&+&\sum_{r+s=m}\nabla^rA\ast\nabla ^sA\ast\nabla^mA.
\end{eqnarray*}
We assume when $l \leq m$, we have $|\nabla^lA|^2\leq C_l(T)$. Then
for $n=m+1$, we have
$$ \frac{\partial}{\partial t}|\nabla^{m+1}A|^2\leq h\triangle|\nabla^{m+1}A|^2+C(T)_1\cdot(|\nabla^{m+1}A|^2+1).$$
We choose $f=|\nabla^{m+1}A|^2+N|\nabla^mA|^2$, where $N$ is a
constant large enough. Then
\begin{eqnarray*}
\frac{\partial }{\partial t}f&\leq&
h\triangle|\nabla^{m+1}A|^2+C(T)_1\cdot(|\nabla^{m+1}A|^2+1)
\\&+&Nh\triangle|\nabla^mA|^2-2hN|\nabla^{m+1}A|^2+C(T)_2
\\&\leq&h\triangle f-C(T)_3|\nabla^{m+1}A|^2+C(T)_4
\\&=&h\triangle f-C(T)_3(f-N|\nabla^mA|^2)+C(T)_5
\\&\leq&h\triangle f-C(T)_3f+C(T)_6.
\end{eqnarray*}
Thus $f\leq C_T$.
\end{proof}
\begin {coro}
$T=+\infty$.
\end {coro}
\section {Proof of the main theorem}
Since the upper bound we derived above is a constant depending on
$T$, $|A|^2$ may be unbounded when $t$ tends to infinity. We will
show that this will not happen and the initial hypersurface
converges to a constant mean curvature surface.
\begin {thm}
The mean curvature $H$ of the evolving surfaces converge to a
constant as $t\rightarrow \infty$.
\end {thm}
\begin{proof} Since $\frac{d}{dt} Vol(E_t)=\int _{M_t}(1-hH)d\mu_t.$
\\Then we have
\begin{eqnarray*}
\int_0^{\infty}\frac{d}{dt} Vol(E_t)dt&=&\int_0^{\infty}\int
_{M_t}(1-hH)d\mu_tdt\\&=&Vol(E_{\infty)}-Vol(E_0)\leq C.
\end{eqnarray*}
Therefore,$$\lim_{t\rightarrow \infty}\int _{M_t}(1-hH)d\mu_t=0.$$
Thus,$$\lim_{t\rightarrow \infty}\int_{M_t}d\mu_t=\lim_{t\rightarrow
\infty}\frac{(\int_{M_t Hd\m_t})^2}{\int_{M_t}H^2d\m_t}. $$ Then by
Cauchy-Schwarz inequality, $H=C$ for some constant.
\end{proof}
\begin{proof}
Rotationally symmetric hypersurfaces of constant mean curvature in
$\mathbb{R}^{n+1}$ are plane, sphere, cylinder, catenoid, unduloid
and nodoid, they are known as the Delaunay surfaces (see [8]). Our
boundary conditions excludes the possibilities of plane, sphere,
catenoid and nodoid. In [2] (see Section 1), Athanassenas use the
condition $|M_0|\leq \frac{V}{d}$ to exclude the existence of
unduloids in $G$. So our possibility can only be the cylinder. Thus
the Main Theorem is proved.
\end{proof}

\ack{The result of this paper is surveyed under the
supervision of Professor Sheng WeiMin at Zhejiang University. The
author would like to express sincere gratitude to Professor Sheng
WeiMin for his guide in these years. His careful reading and
comments have led to an improvement of this paper.}


\end{document}